\documentclass{amsart}
\usepackage{graphicx}
\usepackage{amsrefs}
\usepackage[mathscr]{eucal}
\usepackage{amssymb}
\usepackage{amsmath}

\newtheorem{theorem}{Theorem}
\newtheorem{lemma}[theorem]{Lemma}

\newtheorem{proposition}[theorem]{Proposition}

\begin{document}

\title[Trigonometric polynomials]{When is a trigonometric polynomial not a trigonometric polynomial?}
\thanks{This article was inspired by a classroom discussion that arose during a linear algebra course taught by the first author in Fall 2010.}

\author[J. Borzellino]{Joseph E. Borzellino}
\address{Department of Mathematics, California Polytechnic State University, San Luis Obispo, CA 93407 USA}
\email{jborzell@calpoly.edu}

\author[M. Sherman]{Morgan Sherman}
\address{Department of Mathematics, California Polytechnic State University, San Luis Obispo, CA 93407 USA}
\email{sherman1@calpoly.edu}

\maketitle

Mathematicians generally accept \cites{MR0519124, MR924157} that the term \emph{trigonometric polynomial}, refers to a function $f(t)\in C^\infty(\mathbb{R},\mathbb{C})$ which can be expressed in the following form:
\begin{equation*}\label{trigpoly}
	f(t) = \sum_{n=0}^k a_n \cos (nt) + \sum_{n=1}^k b_n \sin (nt)
\end{equation*}
for some nonnegative integer $k$ and complex numbers $a_0, \ldots, a_k, b_1, \ldots, b_k\in\mathbb{C}$. Trigonometric polynomials and their series counterparts, the
\emph{Fourier series}, play an important role in many areas of pure and applied mathematics and are likely to be quite familiar to the reader.
When reflecting on the terminology, however, it is reasonable to wonder why the term \emph{trigonometric polynomial} is not reserved for a function $g(t)$ of the form
\begin{equation*}\label{alttrigpoly}
	g(t) = \sum_{n=0}^k \alpha_n \cos^n (t) + \sum_{n=1}^k \beta_n \sin^n(t)
\end{equation*}
for some nonnegative integer $k$ and complex numbers $\alpha_0, \ldots, \alpha_k, \beta_1, \ldots, \beta_k\in\mathbb{C}$. In this article, we shall refer to these as \emph{naive trigonometric polynomials}.

It is clear that each of these sets of trigonometric polynomials form a subspace of $C^\infty(\mathbb{R},\mathbb{C})$, but
at first glance it isn't obvious that these two subspaces are in fact distinct.  
An exercise in trigonometric identities, however, shows that any naive trigonometric polynomial can be written as a (standard) trigonometric polynomial. We leave the straightforward details to the reader.

\begin{proposition}\label{NaiveToStd} Any naive trigonometric polynomial can be written as a (standard) trigonometric polynomial.
\end{proposition}

What is curious (and perhaps less clear) is that not all trigonometric polynomials can written as naive trigonometric polynomials. It is known that $\cos(nt) = T_n(\cos t) $ where $T_n$ is the $n$-th \emph{Chebyshev polynomial of the first kind}, and thus $\cos(nt)$ is expressible as a naive trigonometric polynomial.  For the sine terms, one has the identity:  $\sin(nt) =  (\sin t)\cdot U_{n-1}(\cos t)$ where $U_{n-1}$ is the $(n-1)$-st \emph{Chebyshev polynomial of the second kind}.  Both $T_n(x)$ and $U_n(x)$ are polynomials of degree $n$ which are even or odd functions, involving only even or odd powers of $x$, according as $n$ is even or odd. It follows that $\sin((2k+1)t)$ can be expressed as a naive trigonometric polynomial as well. For the definition and properties of the Chebyshev polynomials, see  \cite{MR1937591}. Thus, the resolution of our problem is reduced to showing that polynomials of the form $\sin(2kt)$ cannot be expressed as naive trigonometric polynomials. Without too much work, one can provide an elementary argument substantiating this claim and we invite the interested reader to produce one. Our goal here is to prove this fact as an application of the celebrated theorem of B\'ezout from algebraic geometry. The following version of B\'ezout's theorem is suitable for our needs and follows as an immediate consequence of the usual statement of B\'ezout's theorem which involves zero sets of homogeneous polynomials in the complex projective plane $\mathbb{C}P^2$ (see \cites{MR1417938,MR1042981, MR513824}).

\begin{theorem}[B\'ezout's theorem (weak form)]\label{Bezout} Let $p(x,y)$, $q(x,y) \in \mathbb{C}[x,y]$ be complex polynomials of degree $m$, $n$ with $\gcd\{p,q\}=1$. The number of intersection points in $\mathbb{C}^2$ of the two curves $p(x,y)=0$, $q(x,y)=0$ is at most $mn$. 
\label{Bezout thm}
\end{theorem} 

\begin{proposition}\label{sine2tnotnaive} The function $\sin (2kt)$ cannot be represented as a naive trigonometric polynomial.
\end{proposition}

Our proof of proposition~\ref{sine2tnotnaive} uses the following lemma, which states that polynomial relations between $\cos(t)$ and $\sin(t)$ are all consequences of the Pythagorean identity $\cos^2(t) + \sin^2(t) = 1$. For $p(x,y)\in\mathbb{C}[x,y]$, we let $\bigl(p(x,y)\bigr)$ denote the ideal generated by $p(x,y)$.

\begin{lemma}\label{polyrelations}  Let $R(x,y) \in \mathbb{C}[x,y]$.  Then $R\bigl(\cos(t), \sin(t)\bigr) =0$ for every $t\in\mathbb{R}$ if and only if $R(x,y) \in \bigl(x^2+y^2-1\bigr)$.
\end{lemma}

\begin{proof}
If $R(x,y) \in \bigl(x^2+y^2-1\bigr)$ then clearly $R\bigl(\cos(t), \sin(t)\bigr) \equiv 0$.  Conversely, suppose that $R\bigl(\cos(t), \sin(t)\bigr)$ is identically zero and suppose that $R(x,y) \not\in (x^2+y^2-1)$.  Since $x^2+y^2-1$ is \emph{irreducible}, by B\'ezout's theorem, it follows that the number of intersection points of the curves
$R(x,y)=0$ and $x^2+y^2-1=0$ is at most $2\cdot\mathrm{deg}(R)$. This implies that $\cos(t) $ and $\sin (t)$ take on only finitely many values as $t$ ranges over all real numbers, an obvious contradiction.  
\end{proof}

We are now ready to prove proposition~\ref{sine2tnotnaive}.

\begin{proof}[Proof of proposition~\ref{sine2tnotnaive}] Suppose that there are single-variable complex polynomials $P, Q$ such that 
$\sin(2kt) = P(\cos (t)) + Q(\sin (t))$ for all $t\in \mathbb{R}$. Using the identity $\sin(2kt) = (\sin t)\cdot U_{2k-1}(\cos t)$ we deduce that $(\sin t)\cdot U_{2k-1}(\cos t) = P(\cos(t)) + Q(\sin(t))$.  From lemma~\ref{polyrelations}, it follows that $U_{2k-1}(x)y-P(x)-Q(y) \in \bigl(x^2+y^2-1\bigr)$. That is, there is a polynomial $S(x,y)$ such that 
$$U_{2k-1}(x)y-P(x)-Q(y) = S(x,y)(x^2+y^2-1).$$
Let $x^\alpha y^\beta$ be any monomial of $S$.  If $\alpha>0$ then the right hand side has a term $x^\alpha y^{2+\beta}$ which is not matched by any term on the left hand side.  So every monomial of $S$ has the form $y^\beta$. However, if $\beta >0$ then the right-side term $x^2y^\beta$ is not matched on the left side since $U_{2k-1}(x)$ is an odd polynomial.  Thus, $S(x,y)$ is reduced to a constant which is clearly impossible.
\end{proof}

We conclude by noting that the argument of lemma~\ref{polyrelations} using B\'ezout's theorem can be applied to show that any pair of functions $f(t)$, $g(t)$, each with a range of infinite cardinality, can satisfy essentially at most one polynomial relation. Thus, for example, any polynomial relation of $x = \cosh t$ and $y = \sinh t$ must be a consequence of $x^2-y^2=1$.


\bib{MR0519124}{book}{
  author={Burden, Richard L.},
  author={Faires, J. Douglas},
  author={Reynolds, Albert C.},
  title={Numerical analysis},
  publisher={Prindle, Weber \& Schmidt, Boston, Mass.},
  date={1978},
  pages={ix+579},
  isbn={0-87150-243-7},
  review={\MR {0519124 (58 \#24827)}},
}

\bib{MR1417938}{book}{
  author={Cox, David},
  author={Little, John},
  author={O'Shea, Donal},
  title={Ideals, varieties, and algorithms},
  series={Undergraduate Texts in Mathematics},
  edition={2},
  note={An introduction to computational algebraic geometry and commutative algebra},
  publisher={Springer-Verlag},
  place={New York},
  date={1997},
  pages={xiv+536},
  isbn={0-387-94680-2},
  review={\MR {1417938 (97h:13024)}},
}

\bib{MR1042981}{book}{
  author={Fulton, William},
  title={Algebraic curves},
  series={Advanced Book Classics},
  note={An introduction to algebraic geometry; Notes written with the collaboration of Richard Weiss; Reprint of 1969 original},
  publisher={Addison-Wesley Publishing Company Advanced Book Program},
  place={Redwood City, CA},
  date={1989},
  pages={xxii+226},
  isbn={0-201-51010-3},
  review={\MR {1042981 (90k:14023)}},
}

\bib{MR1937591}{book}{
  author={Mason, J. C.},
  author={Handscomb, D. C.},
  title={Chebyshev polynomials},
  publisher={Chapman \& Hall/CRC, Boca Raton, FL},
  date={2003},
  pages={xiv+341},
  isbn={0-8493-0355-9},
  review={\MR {1937591 (2004h:33001)}},
}

\bib{MR924157}{book}{
  author={Rudin, Walter},
  title={Real and complex analysis},
  edition={3},
  publisher={McGraw-Hill Book Co.},
  place={New York},
  date={1987},
  pages={xiv+416},
  isbn={0-07-054234-1},
  review={\MR {924157 (88k:00002)}},
}

\bib{MR513824}{book}{
  author={Walker, Robert J.},
  title={Algebraic curves},
  note={Reprint of the 1950 edition},
  publisher={Springer-Verlag},
  place={New York},
  date={1978},
  pages={x+201},
  isbn={0-387-90361-5},
  review={\MR {513824 (80c:14001)}},
}



\begin{bibdiv} 
\begin{biblist}

\bibselect{ref}
 
\end{biblist} 
\end{bibdiv}
\end{document}